\documentclass{gtpart}   
\usepackage[english]{babel}
\usepackage{amssymb,amsmath,amscd,amstext,amsfonts,amsthm}
\usepackage{pinlabel}
\usepackage[all]{xy}

\title{Minimal algebraic complexes over $D_{4n}$}

\author{W.H. Mannan} 
\givenname{Wajid}
\surname{Mannan}
\address{Mathematics and Statistics\\ Lancaster University\\\newline Lancaster LA\textup{1 4}YF\\United Kingdom
\newline\newline Department of Mathematics\\ University College London\\\newline London WC\textup{1E 6}BT\\United Kingdom}
\email{wajid@mannan.info}

\author{S. O'Shea} 
\givenname{Seamus}
\surname{O'Shea}
\email{s.o'shea@ucl.ac.uk}
\urladdr{http://iris.ucl.ac.uk/iris/browse/profile?upi=SOSHE53}


\numberwithin{equation}{section}
\newtheorem{theorem}{Theorem}[section]

\newtheorem{lemma}[theorem]{Lemma}

\theoremstyle{definition}
\newtheorem{definition}[theorem]{Definition}

\def\co{\colon\thinspace}

\renewcommand {\Z} {\mathbb Z}
\renewcommand {\Q} {\mathbb Q}
\newcommand {\ZG} {\mathbb {Z}[G]}

\newcommand {\ZD} {\mathbb {Z}[D_{4n}]}

\newcommand {\ZC} {\mathbb {Z}[C_{2}]}
\newcommand {\QC} {\mathbb {Q}[C_{2}]}
\newcommand {\Zd} {\mathbb {Z}[D_{8}]}
\newcommand {\QD} {\mathbb {Q}[D_{4n}]}

\makeop{Homo}

\begin{document}

\begin{abstract}    
We show that cancellation of free modules holds in the stable class $\Omega_3(\Z)$ over dihedral groups of order $4n$.  In light of a recent result on realizing $k$--invariants for these groups, this completes the proof that all all dihedral groups satisfy the D(2) property.
\end{abstract}

\maketitle

\section{Introduction}\label{intro}

In 1965 Wall showed  that for $n>2$, if a finite cell complex is cohomologically $n$ dimensional (in the  sense of having no non-trivial cohomology in dimensions above $n$ with respect to any coefficient bundle), then it is in fact homotopy equivalent to an actual $n$ dimensional cell complex \cite{Wall}.  Subsequently it was shown by Swan and Stallings that the only cohomologically 1 dimensional finite cell complexes are disjoint unions of wedges of circles \cite{Stal, Swan1}.  However decades later the case $n=2$ remains a major open problem, known as Wall's D(2)--problem.  

The problem may be phrased in terms of realizing algebraic complexes geometrically \cite{John1, Mann2, Mann4}, or in terms of group presentations \cite{Mann5}.  Cell complexes which potentially offer a counterexample to Wall's D(2)--problem have been postulated \cite{John3, Brid}, though proving that they are counterexamples would appear to require some new obstruction.  Where progress has been made is extending Wall's result to the case $n=2$ for all finite cell complexes with a specified fundamental group.

\begin{definition}
We say that a group $G$ satisfies the {\emph{D(2) property}} if all cohomologically 2 dimensional finite cell complexes with fundamental group $G$ are in fact homotopy equivalent to a 2 dimensional finite cell complex.
\end{definition} 

At present, the D(2) property in known to hold only in comparatively few cases. These include cyclic groups, products of the form $C_{\infty} \times C_n$  \cite{Edwa} and, more relevantly to the present paper, the dihedral groups $D_{4n+2}$  \cite{John, John1}.

The first finite nonabelian non periodic group shown to have the D(2) property was the dihedral group of order 8 \cite{Mann1}.  The methodology adopted in \cite{Mann1} was an instance of the approach developed and laid out in \cite{John1}.  This approach for verifying the D(2) property for a given group $G$ involves two steps.

The first of these is to classify all possible modules over $\ZG$ which may arise as the second homotopy group of a connected 
cohomologically 2 dimensional finite cell complex $X$.  The Hurewicz homomorphism identifies $\pi_2(X)$ with $H_2(\tilde{X})$.  As $X$ is cohomologically 2 dimensional, consideration of the homotopy type of the algebraic cellular chain complex  of $\tilde{X}$ yields  an exact sequence of $\ZG$ modules: 
$$0 \to \pi_2(X)  \to F_2 \to F_1 \to F_0 \to \Z \to 0,$$
where the $F_i$ are finitely generated (henceforth denoted f.g.\!) free \cite{Mann3} modules \cite[Appendix B]{John1}.  Thus by Schanuel's lemma all such $\pi_2(X)$ are stably isomorphic.  

This stable class of modules, denoted $\Omega_3(\Z)$,  is discussed extensively in \cite{John1}, in particular in \S29.  For our purposes it is sufficient to recall that $\Omega_3(\Z)$ may be viewed as a directed tree, with modules represented by vertices and partitioned into levels, which are well-ordered.  If $G$ is finite then the level of a module is determined simply by its  $\Z$--rank.  An element of $\Omega_3(\Z)$ is called {\it minimal} if it occurs at the minimal level.

For modules over$\Zd$ one may use a strong cancellation result of Swan \cite[Theorem 6.1]{Swan2} to deduce that  $\pi_2(X) \cong J \oplus \Zd^r$ for  some integer $r$, where $J$ is the  `standard' minimal element of the stable class $\Omega_3(\Z)$ (see (\ref{res})).  

Having classified all possible modules $K$ which may arise as $\pi_2(X)$ in this way, the second step is to show that for each $K$,  any cohomologically 2 dimensional finite cell complex $X$ with $\pi_2(X)\cong K$, is homotopy equivalent to a 2 dimensional cell complex.  This involves realizing the so called `$k$--invariants' \cite[Chapter 6]{John1}.

In \cite{Mann1} the second step is completed for dihedral groups of order $2^n$, and \newline $K\cong J \oplus \Z[D_{2^n}]^r$ for any $r$.  Recently these methods have been extended to the more general  case of all dihedral groups of order $4n$:

\begin{theorem}\label{shea}
{\rm \cite[Theorem 1.3 and Corollary 2.6]{Shea}} Suppose that $X$ is a cohomologically 2 dimensional finite cell complex with fundamental group $D_{4n}$ and $\pi_2(X)\cong J \oplus \ZD^r$, for some integer $r$.  Then $X$ is homotopy equivalent to a 2 dimensional finite cell complex. 
\end{theorem}

This key result is what our proof is predicated on.  In order to prove that the remaining dihedral groups satisfy the D(2) property, it remains to show that over $\ZD$ any module in $\Omega_3(\Z)$ must be of the form $J \oplus \ZD^r$.  This is our main result.

\bigskip
\noindent {\bf Theorem A} {\sl Over $\ZD$ any module in $\Omega_3(\Z)$ is  of the form $J \oplus \ZD^r$.}

We note that this has been proved for modules  in $\Omega_3(\Z)$ which are geometrically realizable as the second homotopy group of an actual 2--complex \cite{Lati, Hamb}.  However from the point of view of the D(2)--problem, showing the existence of such a 2--complex is a key difficulty, so it is not clear if these methods could be extended to address the D(2)--problem for $D_{4n}$.  

Having proved Theorem A, we may conclude from \cite{John, John1} and Theorem 1.2:

\bigskip
\noindent {\bf Theorem B} {\sl 
All finite dihedral groups satisfy the D(2) property.
}

\noindent {\it Outline of Proof of Theorem A}:  In  \S\ref{minimal} we note that the only modules in $\Omega_3(\Z)$ which may not be of the required form must occur at the minimal level.  We then observe that any such minimal element of $\Omega_3(\Z)$ is the kernel of some surjective map $j\co\ZD^3\to W_2$, where $W_2\subset \ZD^2$ is the image of the `standard' $\partial_2$.

In \S\ref{rest} we consider separately the components $j',j''$ of the map $j$ into the first and second summands of $\ZD^2$.  By Schanuel's Lemma we know that ker$(j'')$ is stably equivalent to ker$(\partial_2'')$, which we denote $M$.  However $M$ contains a free summand, so by the Swan--Jacobinski Theorem we know that ker$(j'')\cong M$.  The kernel of $j$ is just the kernel of the restriction $j'\vert_{{\rm ker}(j'')}$.  The image of this restriction of $j'$ is isomorphic to the module $\ZC$, so we may conclude that any minimal element of $\Omega_3(\Z)$ is in fact the kernel of a surjective homomorphism $M\to \ZC$. 

In \S\ref{isom} we construct isomorphisms between kernels of such maps, using elementary matrices and other means.  In particular we define modules $K_{x,1}$ for integers $x$ and show that any kernel of a surjective homomorphism $M \to \ZC$ is isomorphic to $K_{l,1}$, for $l$ some positive factor of $2n$.

In \S\ref{stabsep} we construct maps ${{\partial_2}^{(l)}}\co \ZD^3 \to \ZD^2$ for $l$ a positive factor of $2n$.   In each case ker$({{\partial_2}^{(l)}})=K_{l,1}$.  Also ${{\partial_2}^{(1)}}={{\partial_2}}$ which has kernel $J$.  It remains to show that the kernels of the remaining ${{\partial_2}^{(l)}}$ are not stably isomorphic to $J$.  

We do this by noting that the kernel of ${{\partial_2}^{(l)}}$ is stably isomorphic to $J$ only if the cokernel of ${{\partial_2}^{(l)}}$ is stably isomorphic to $I$, the augmentation ideal.  As cancellation of free modules holds in the stable class of the augmentation ideal, this is the same as saying that the cokernel of ${{\partial_2}^{(l)}}$ is isomorphic to the augmentation ideal.  

For $l>1$ we show that this is not the case, leaving $J$ as the sole minimal element of $\Omega_3(\Z)$.

\section{Minimal elements of $\Omega_3(\Z)$}\label{minimal}

Fix the presentation $D_{4n}=\langle a,b \vert\,a^{2n}=b^2=1,\,\, aba =b \rangle.$
We work throughout over the integral group ring $\ZD$ and all modules are right modules.  Let $\Sigma$ denote $\sum_{i=0}^{2n-1} a^i\in \ZD$.  Consider the following exact sequence, taken from \cite{Mann1}:
\begin{eqnarray}
0 \to J\longrightarrow \ZD^3 \stackrel{\partial_2}{\longrightarrow} \ZD^2 \stackrel{\partial_1}{\longrightarrow} \ZD \stackrel{\epsilon}{\longrightarrow} \Z \to 0. \label{res}
\end{eqnarray}
Here $\epsilon$ is determined by $\epsilon(1)=1 \in \Z$.   Take basis elements $e_1, e_2\in \ZD^2$ and $E_1, \, E_2, \, E_3 \in \ZD^3$.  Then $\partial_1, \partial_2$ are given by: $$
\begin{array}{cclccccl}
\partial_1 e_1 &=& 1-a,&&&\partial_2 E_1 &=& e_1 \Sigma, \\
\partial_1 e_2 &=&1-b,&&&\partial_2 E_2 &=& e_2 (1+b), \\
&&&&&\partial_2 E_3 &=&  e_1(1+ba)+e_2(a-1),
\end{array}$$
and $J$ is the kernel of $\partial_2$.

With respect to the basis $\{E_1,\, E_2,\,E_3\}$ and $\{e_1, \, e_2\}$,  we have:

\begin{eqnarray}
\partial_2=\left[ \begin{array}{ccc} \Sigma & 0 & 1+ba \\ 0 & 1+b & a-1 \end{array} \right] \label{d2}
\end{eqnarray}

We know that $J$ has minimal $\Z$--rank in $\Omega_3(\Z)$ \cite[Proposition 3.2]{Mann1}.  From the Swan--Jacobinski Theorem \cite[Theorem 15.1]{John1}, we know that the only non-minimal modules in $\Omega_3(\Z)$ are ones of the form $J\oplus \ZD^r$.  It remains to show that up to isomorphism, $J$ is the only minimal element of $\Omega_3(\Z)$.

Let $W_2$ denote the image of $\partial_2$ so we have a short exact sequence:
\begin{eqnarray}0\to J \to \ZD^3 \to W_2\to 0. \label{Jembed}\end{eqnarray}  Let $K$ be a minimal element of $\Omega_3(\Z)$.  That is let $K$ be a module stably equivalent to $J$ and of the same $\Z$--rank.  Thus $K \oplus F \cong J \oplus F'$ for f.g.\!\! free modules $F$ and $F'$ of the same rank.  With respect to this isomorphism, modifying (\ref{Jembed}) by direct summing with the identity map $1_{F'}\co F' \to F'$ yields the short exact sequence:
\begin{eqnarray}0 \to K\oplus F \to \ZD^3 \oplus F' \to W_2\to 0.\label{Kembed}\end{eqnarray}

\begin{lemma} The inclusion in {\rm (\ref{Kembed})}, $\iota\co F \to \ZD^3 \oplus F'$, splits. \end{lemma}

\begin{proof}
The cokernel of $\iota$ is torsion free and finitely generated.  To deduce that $\iota$ splits, it is sufficient to note that $F$ is injective relative to the tame class of f.g.\!\! torsion free $\ZD$ modules (see \cite[p.69]{John1} for a discussion of relatively injective modules and tame classes).  This follows from $F$ being projective and property $\mathcal{T}3$ in \cite[p.70]{John1}.  
\end{proof}

Thus  $\ZD^3 \oplus F' \cong S\oplus F$, for some stably free module $S$.  As $\ZD$ satisfies the Eichler property, any
stably free module over it must in fact be free \cite[Proposition 15.7]{John1}).  As $F, F'$ have the same $\Z$--rank, we know that $S\cong \ZD^3$.  Hence we may write (\ref{Kembed}) as:
$$0 \to K \oplus F \stackrel{i'}{\longrightarrow} \ZD^3 \oplus F {\longrightarrow} W_2 \to 0,$$
where $i'$ restricts to the identity on $F$.  Restricting to submodules of this sequence, we get the exact sequence of modules:

$$0 \to K \stackrel{i}{\longrightarrow} \ZD^3  \stackrel{j}{\longrightarrow} W_2 \to 0.$$

We may conclude:

\begin{lemma} \label{kernel}
Any minimal element of $\Omega_3(\Z)$ occurs as the kernel of some surjective map $\ZD^3  \stackrel{j}{\longrightarrow} W_2$.
\label{minisker}\end{lemma}

Thus in order to prove our main result (Theorem A), it suffices to show that any kernel of a surjective map $\ZD^3  \stackrel{j}{\longrightarrow} W_2$, is isomorphic to $J$.

\section{Restricting to submodules} \label{rest}

In this section we take an arbitrary surjection $\ZD^3  \stackrel{j}{\longrightarrow} W_2$, and find submodules of $\ZD^3$ and $W_2$ such that $j$ restricts to a surjection between these submodules, with the same kernel as $j$.  In the next section we classify all possible kernels of surjective maps between these submodules.  Thus we will have a comprehensive list of all modules which may be a minimal  element of $\Omega_3(\Z)$.  In the final section we will show that only one of these modules is in $\Omega_3(\Z)$.

\bigskip
Recall from (\ref{d2}) the map $\partial_2\co \ZD^3\to \ZD^2$: \begin{eqnarray*}
\partial_2=\left[ \begin{array}{ccc} \Sigma & 0 & 1+ba \\ 0 & 1+b & a-1 \end{array} \right] 
\end{eqnarray*}
Let  $\partial_2',\partial_2''$ denote the components of $\partial_2$, onto summands of $\ZD^2$ generated by $e_1, e_2$ respectively.  Let $M \subset \ZD^3$ denote the kernel of $\partial_2''$.

\begin{lemma}
The module $M$ is generated by{\rm:}
\begin{eqnarray*}
w_1&=&E_1,\\
w_2&=&E_2(1-b),\\
w_3&=&E_3\Sigma,\\
w_4&=&E_2(a-1)-E_3(1-ba).
\end{eqnarray*}\label{doubleus}
\end{lemma}

\begin{proof}
Consider the sequence (\ref{res}) and apply the automorphism of $\ZD$ given by $a \mapsto a,\,b\mapsto -b$ to the elements of the matrices representing the maps $\partial_1,\partial_2$.  The new sequence is still exact, so we see that $w_2,w_3,w_4$ generate all relations between $1+b$ and $a-1$.
\end{proof}

Note that applying $\partial_2'$ to the elements $w_1,w_2,w_3,w_4$ results in elements of $\Sigma\ZD$.  In particular $\partial_2'(w_1)=\Sigma$.  We may conclude:

\begin{lemma} The elements of $W_2$ which have second component 0, are generated by{\rm:} $$\left(\begin{array}{c}\Sigma\\0\end{array}\right).$$ \label{firstcomp}
\end{lemma}

\begin{proof}
An element of $\ZD^2$ with second component $0$ lies in $W_2$ if and only if it is in the  image of $M$ under $\partial_2$.
\end{proof}

At this point we note that $D_{4n}$ has a normal subgroup generated by $a$ and taking the quotient by this normal subgroup gives a surjective group homomorphism $D_{4n}\to C_2$.  We abuse notation by using elements of $D_{4n}$ to refer to their images in $C_2$. With respect to this homomorphism, we may regard any $\ZC$ module $L$ as a $\ZD$ module by implicitly identifying $L$ with the tensor product $L \otimes_{\ZC}\ZD$.  Thus we have that as $\ZD$ modules, $\Sigma \ZD \cong \ZC$.

Now consider an arbitrary surjective homomorphism $j\co \ZD^3 \to W_2$.  Again let $j',j''$ denote the first and second components of $j$.  Let $I^T$ denote the submodule of $\ZD$ generated by $a-1,b+1$.

\begin{lemma}
The kernel of $j''$ is isomorphic to $M$. \label{Misker}
\end{lemma}

\begin{proof}
We have exact sequences:
\begin{eqnarray*}
0 \to M \to \ZD^3 \stackrel {\partial_2''} \to I^T\to 0,\label{firstker}\\
0 \to {\rm ker}(j'') \to \ZD^3 \stackrel {j''} \to I^T\to 0.\label{secker}
\end{eqnarray*}

Thus by Schanuel's Lemma ker$(j'')$ is stably equivalent to $M$.  Clearly they have the same $\Z$--rank.  Note that $M$ contains a free summand generated by $w_1$, so we may apply  the Swan--Jacobinski Theorem \cite[Theorem 15.1]{John1} to deduce ker$(j'')\cong M$.
\end{proof}

\begin{lemma}
Any minimal element of $\Omega_3(\Z)$ is isomorphic to the kernel of a surjective homomorphism $M \to \ZC$. \label{MtoZC2}
\end{lemma}

\begin{proof}
From Lemma \ref{minisker} we know that any minimal element of $\Omega_3(\Z)$ is isomorphic to the kernel of some surjective homomorphism $j\co \ZD^3\to W_2$.  This is the kernel of $j'\vert_{{\rm ker}j''}$. By Lemma \ref{Misker} we may identify ker$(j'')$ with $M$ and by Lemma \ref{firstcomp} the image of  $j'\vert_{{\rm ker}j''}$ is $\Sigma\ZD \cong \ZC$.
\end{proof}

The action of $a$ on $\ZC$ is trivial so the kernel of any surjective map 
$s\co M \to \ZC$, must contain $M(a-1)$.  Thus $s$ is just the quotient map $M\to M/M(a-1)$ composed with a $\ZC$--linear map $M/M(a-1)\to \ZC$.  We seek to better understand the $\ZC$ module $M/M(a-1)$.

Let $\Z^T$ denote the $\ZD$ module whose underlying abelian group
is isomorphic to the integers, and on which $a$ acts trivially and $b$ acts as
multiplication by $-1$.  Let $\Q^T$ denote $\Z^T \otimes \Q$.

\begin{lemma}{$M/M(a-1)$ has (torsion free) $\Z$--rank 5.}\label{rank5}
\end{lemma}

  \begin{proof} We have an exact sequence:
$$0 \to M {\longrightarrow} \ZD^3 \stackrel{\partial_2''}{\longrightarrow}
 \ZD \to \Z^T \to 0.$$
Tensor this sequence with $\Q$ and 
 apply "Whitehead's trick" to get:
$$M \otimes \Q \oplus \QD\cong \QD^3 \oplus \Q^T.$$  Canceling we get $M \otimes \Q \cong \QD^2 \oplus \Q^T$ and 
$$(M/M(a-1))\otimes \Q \cong \QC^2\oplus\Q^T.$$ 
This has $\Q$--rank 5.
 \end{proof}

If $m_1,m_2\in M$ differ by an element of $M(a-1)$ we write $m_1 \sim m_2$.

\begin{lemma} We have an isomorphism of $\ZD$ modules{\rm:} $$M/M(a-1)\cong \ZC \oplus\Z^T\oplus\Z^T\oplus \Z.$$  The summands of $\ZC,\Z^T, \Z^T,\Z$ are generated respectively by the images of{\rm:} $$w_1,w_2,w_4,w_3+w_4n.$$ \label{baction}
\end{lemma}

  \begin{proof}  By Lemma \ref{doubleus} we know that the above four elements generate $M$ and hence $M/M(a-1)$ over $\ZD$.  Clearly the image of $w_1$ generates a copy of $\ZC$.  We must verify the action of $b$ on the images of $w_2,w_4,w_3+w_4n$ by showing:\begin{eqnarray}w_2(1+b),\,\,w_4(1+b),\,\, (w_3+w_4n)(1-b)\,\,\in M(a-1).\label{bactionid}\end{eqnarray}
As the action of $a$ on $M/M(a-1)$ is trivial, we then have that $M/M(a-1)$ is some quotient of $\ZC \oplus\Z^T\oplus\Z^T\oplus \Z$, with each summand generated by the required image.  However Lemma \ref{rank5} tells us that there can be no further $\ZD$--linear relations between the four images, and we are done.

We have:
\begin{eqnarray*}
w_2(1+b)&=&0\in M(a-1),\\
w_4(1+b)&\sim& w_4(1+ba)=w_2(a-1)\in M(a-1),\\
(w_3+w_4n)(1-b)&=&-w_4\Sigma +w_4(1-b)n\\&\sim& -w_4\Sigma +w_4(1-b)n+w_4(1+b)n\\&=&w_4(2n-\Sigma)\in M(a-1).
\end{eqnarray*}
\end{proof}

Recall that any surjection $s\co M \to \ZC$ is the composition of  the quotient map $M\to M/M(a-1)$ with some $\ZC$--linear map $M/M(a-1) \to \ZC$, which completely determines $s$. 

Let integers $s_1,s_2,s_3$ denote the multiples of $1-b\in\ZC$ that $s$ maps $w_1(1-b),$ $w_2, w_4$ to respectively.  

Let integers $s_4,s_5$ denote the multiples of $1+b\in \ZC$ that $s$ maps $w_1(1+b),w_3+w_4n$ to respectively.

We have that $s$ is completely described by these five integers (as $w_12=w_1(1-b)+w_1(1+b)$).  We write $s=[(s_1,s_2,s_3),(s_4,s_5)]$.

\begin{lemma}
The surjective maps $s\co M \to \ZC$ are precisely the maps  $s=[(s_1,s_2,s_3),(s_4,s_5)]$, for coprime integers $s_1,s_2,s_3$ and coprime integers $s_4,s_5$ with $s_1, s_4$ odd. \label{cond}
\end{lemma}

\begin{proof}
Note that $w_2,w_4,w_3+w_4n$ must all map to elements of even augmentation in $\ZC$.  Thus for $s$ to be surjective, $s(w_1)$ must have odd augmentation.  Thus $s_1,s_4$ must be odd.  Further for $s$ to be surjective, we must also have $s_1,s_2,s_3$ coprime, and $s_4,s_5$ coprime.

Conversely given any such integers we have a surjective map $s=[(s_1,s_2,s_3),(s_4,s_5)]$, where $$s(w_1)=(s_1+s_4)/2+b(s_4-s_1)/2,$$ noting that $s_1,s_4$ are both odd, so the division is permissible.
\end{proof}

Thus all potential minimal elements of $\Omega_3(\Z)$ are parametrized by quintuples of integers satisfying these coprimality and oddness conditions.

\section{Isomorphisms between kernels} \label{isom}

Now we know that any minimal element of $\Omega_3(\Z)$ is isomorphic to the kernel of a map $s=[(s_1,s_2,s_3),(s_4,s_5)]$ satisfying the conditions of Lemma \ref{cond}.  In this section we construct isomorphisms between the kernels of maps corresponding to different quintuples of integers.  Ultimately we will show that without loss of generality we may assume that $s=[(1,0,0),(1,l)]$, where $l$ is a positive factor of $2n$.

For $i=1,2,3,4$ define automorphisms $\phi_i\co \ZD^3\stackrel \sim \to\ZD^3$ by:
\begin{eqnarray*}
\phi_1\co E_2 &\mapsto& E_2+w_1,\\
\phi_2\co  E_3 &\mapsto& E_3+w_1,\\
\phi_3\co E_1 &\mapsto& E_1+w_2,\\
\phi_4\co  E_1 &\mapsto& E_1+w_4,\\
\end{eqnarray*}
and in each case $\phi_i$ maps the other two generators (out of $E_1,E_2,E_3$) to themselves.  As each $\phi_i$ differs from the identity by a map which factors through the inclusion $M\to \ZD^3$, the following diagram commutes:
$$\begin{array}{ccllccccc}
0& \to & M &\longrightarrow & \ZD^3& \stackrel {\partial_2''} \longrightarrow &I^T&\to& 0\\
&&\,\,\downarrow\phi_i\vert_M&&\downarrow\phi_i&&\,\,\downarrow 1
\\
0& \to & M &\longrightarrow & \ZD^3& \stackrel {\partial_2''} \longrightarrow &I^T&\to& 0\\
\end{array}
$$
and the restriction of $\phi_i$ to $M$ is an isomorphism $\phi_i\vert_M\co M \stackrel \sim \to M$.

\noindent For $i=1,2,3,4$ define linear maps $\phi_i'\co \Z^3 \to \Z^3$ and $\phi_i'' \co \Z^2 \to \Z^2$ by:
\begin{eqnarray*}
\phi_1'(s_1,s_2,s_3)=(s_1,s_2+s_1,s_3), \,\,\,&\qquad\qquad&\phi_1''(s_4,s_5)=(s_4,s_5),\\
\phi_2'(s_1,s_2,s_3)=(s_1,s_2,s_3-s_1), \,\,\,&\qquad\qquad&\phi_2''(s_4,s_5)=(s_4,s_5+ns_4),\\
\phi_3'(s_1,s_2,s_3)=(s_1+2s_2,s_2,s_3), &\qquad\qquad&\phi_3''(s_4,s_5)=(s_4,s_5),\\
\phi_4'(s_1,s_2,s_3)=(s_1+2s_3,s_2,s_3), &\qquad\qquad&\phi_4''(s_4,s_5)=(s_4,s_5).\\
\end{eqnarray*}

\begin{lemma}\label{operations}
For $i=1,2,3,4$ and integers $s_1,s_2,s_3,s_4,s_5$ satisfying the conditions of Lemma \ref{cond} let{\rm:}$$
s=[(s_1,s_2,s_3),(s_4,s_5)], \qquad \qquad s'=[\phi_i'(s_1,s_2,s_3),\phi_i''(s_4,s_5)].$$
Then the following diagram commutes{\rm:}
$$\begin{array}{ccllccccc}
 M& \stackrel {s'} \longrightarrow &\ZC&\\
\qquad\,\,\downarrow\phi_i\vert_M&&\,\,\,\,\downarrow 1
\\
 M& \stackrel {s} \longrightarrow &\ZC\\
\end{array}
$$
so the isomorphism $\phi_i\vert_M$ restricts to an isomorphism between the kernels of $s$ and $s'$.
\end{lemma}

\begin{proof}
Direct calculation (recalling (\ref{bactionid})) shows that the maps $s\phi_i\vert_M$ and $s'$ yield the same result when evaluated on  $w_1(1-b),$  $w_2,w_4,w_1(1+b),w_3+w_4n\in M$.  
\end{proof}

Thus the isomorphism class of the kernel of a map $M \to \ZC$ corresponding to a quintuple of integers is unaltered by applying the operations $\phi_i',\phi_i''$ to the integers.

\begin{lemma}
Any minimal element of $\Omega_3(\Z)$ is isomorphic to the kernel of a map $s\co M\to \ZC$, where $s=[(1,0,0), (r,x)]$ with $r$ an odd integer and $x,r$ coprime. \label{xr}
\end{lemma}

\begin{proof}
From Lemma \ref{MtoZC2} and Lemma \ref{cond} we know that any minimal element of  $\Omega_3(\Z)$ is isomorphic to the kernel of some map $s\co M\to \ZC$, with $s=[(s_1,s_2,s_3), (s_4,s_5)]$ satisfying the conditions of Lemma \ref{cond}.  Further by Lemma \ref{operations} we are free to apply the operations $\phi_i'$ and $\phi_i''$ to the five integers, without altering the isomorphism class of the kernel.  Note also that these operations preserve the coprimality and oddness conditions of Lemma \ref{cond}.  It remains to show that by repeatedly applying the operations $\phi_i'$, we may reduce $(s_1,s_2,s_3)$ to $(1,0,0)$.

We apply a version of Euclid's algorithm first to $s_1,s_2$, using the operations $\phi_1',\phi_3'$.  Note that whilst $\phi_1'$ allows us to divide $s_2$ by $s_1$ and replace $s_2$ with the remainder, $\phi_3'$ only allows us to divide $s_1$ by $2s_2$ and replace $s_1$ with the remainder.  However this is still sufficient for guaranteeing that the remainder has modulus less than or equal to the modulus of $s_2$ (allowing negative remainders). 

Thus eventually we will have $s_2=0$ (we cannot have $s_1=0$ as $s_1$ remains odd).  We then repeat the process with the operations $\phi_2'$ and $\phi_4'$ to obtain $s_3=0$.  As $(s_1,s_2,s_3)$ remain coprime under the all operations $\phi_i'$, we must have $s_1=\pm1$.  
If $s_1=-1$, then we may apply $\phi_2'\phi_4'\phi_2'$ to $(-1,0,0)$ to get $(1,0,0)$.
\end{proof}

Now let $s=[(1,0,0), (r,x)]$ with $r$ an odd integer and $x,r$ coprime.  Recall that $s$ is the composition of the quotient map $M \to M/M(a-1)$ with a map $\bar{s}\co M/M(a-1) \to \ZC$.

\begin{lemma}
The kernel of $\bar{s}$ is generated over $\Z$ by the images of{\rm:} $$w_2,\,\,w_4,\,\, w_1(1+b)x-(w_3+w_4n)r.$$
\end{lemma}

\begin{proof}
Recall from Lemma \ref{baction} that $M/(M(a-1) \cong \ZC \oplus \Z^T \oplus\Z^T \oplus \Z$, with the summands generated by $w_1, w_2,w_4,w_3+w_4n$ respectively.  Also recall that $s$ maps the last three of these to elements of even augmentation in $\ZC$, and maps $w_1$ to an element of odd augmentation in $\ZC$.  Thus given an element $\zeta$ of the kernel of $\bar{s}$, the $\ZC$ coefficient on the image of $w_1$ must have even augmentation.  

Thus $\zeta$ is in the span of the images of $w_1(1-b), w_2, w_4$ and  $w_1(1+b),w_3+w_4n$ and we may write $\zeta=\zeta_1+\zeta_2$ with $\zeta_1$ in the span of the first three and $\zeta_2$ in the span of the latter two.  Then $\bar{s}(\zeta)=0$ if and only if $\bar{s}(\zeta_1)=0$ and $\bar{s}(\zeta_2)=0$.  

As $s=[(1,0,0), (r,x)]$, we have $\zeta_1$ in the span of the images of $w_2,w_4$ and $\zeta_2$ in the span of  the image of $w_1(1+b)x-(w_3+w_4n)r$.
\end{proof}

Fix $\hat{K}=\langle M(a-1),w_2, w_4\rangle\subset M$ and let $$\zeta_{x,r}=w_1(1+b)x-(w_3+w_4n)r.$$  Also let $K_{x,r}$ denote the kernel of $s$.  We may conclude:

\begin{lemma}
We have  $K_{x,r}=\langle \hat{K},\zeta_{x,r}\rangle\subset M$.
\end{lemma}

\newpage
\begin{lemma}
For integers $p,q$ not both $0$, the isomorphism class of $\langle \hat{K},\zeta_{p,q}\rangle$ is determined by the congruence classes of $p$ modulo $2n$ and $q$ modulo $2$.
\end{lemma}

\begin{proof}
We have: \begin{eqnarray*}
\zeta_{p+2n,q}&=&\zeta_{p,q}+w_1(1+b)2n\,\,\sim\,\, \zeta_{p,q}+E_1(1+b)\Sigma,\\
\zeta_{p,q+2}&=&\zeta_{p,q}-(w_3+w_4n)2\,\,\sim\,\, \zeta_{p,q}-w_32-w_4\Sigma=\zeta_{p,q}-E_3(1+b)\Sigma.
\end{eqnarray*}

Each of $\langle \hat{K},\zeta_{p,q}\rangle,\langle \hat{K},\zeta_{p+2n,q}\rangle,\langle \hat{K},\zeta_{p,q+2}\rangle$ is isomorphic to $\hat{K}\oplus \Z$ as an abelian group, where the generator $t\in M$ of the summand $\Z$ is $$\zeta_{p,q},\,\, \zeta_{p,q}+E_1(1+b)\Sigma,\,\,\zeta_{p,q}-E_3(1+b)\Sigma$$ respectively.

In each case, the isomorphism class of the module is then determined by the action of $\ZD$ on $t$, which is given by the values of $t(a-1),t(b-1)\in M(a-1)$.  However these values are clearly the same in all three cases.
\end{proof}

Note that as $r$ is always odd (Lemma \ref{xr}) we have in general $K_{x,r} \cong K_{x,1}$ (clearly the integers $x,1$ are always coprime).  Thus we now have that any minimal element of $\Omega_3(\Z)$ is isomorphic to $K_{x,1}$ for some congruence class of integer $x$ modulo $2n$.  We complete this section by showing that multiplying $x$ by a unit in $\Z/2n\Z$ does not alter the isomorphism class of $K_{x,1}$.

Let $r$ be a positive integer coprime to $2n$ and let the positive integer $u$ satisfy $ur+\lambda 2n=1$, for some integer $\lambda$.  Let $\alpha,\beta$ be given by:
$$
\alpha=\sum_{i=0}^{r-1} a^i,\qquad\qquad \beta=\sum_{i=0}^{u-1} a^{ri}.$$  Clearly we have: \begin{eqnarray}\alpha\beta=\beta\alpha=1-\lambda\Sigma. \label{inverse}\end{eqnarray}

Let $\psi_r,\psi_r'\co \ZD^3\to \ZD^3$ be the maps acting as the identity on $E_2,E_3$ and mapping:
\begin{eqnarray*}\psi_r \co E_1 &\mapsto& E_1 \alpha\\\psi_r'\co  E_1 &\mapsto& E_1 \beta.\\
\end{eqnarray*}

\newpage
\begin{lemma}
The map $\psi_r$ restricts to an isomorphism $K_{x,1} \stackrel\sim\to K_{rx,1}$.
\end{lemma}

\begin{proof}
From (\ref{inverse}) we know that $\psi_r\psi_r'$ and $\psi_r'\psi_r$ restrict to the identity on $M(a-1)$.  Thus $\psi_r$ restricts to an isomorphism $M(a-1) \stackrel \sim \to M(a-1)$.  Further: $$\psi_r(w_2)=w_2,\qquad\qquad \psi_r(w_4)=w_4,$$ so $\psi_r$ restricts to an isomorphism $\hat{K}\stackrel \sim \to\hat{K}$.  Finally we note:
\begin{eqnarray*}
\psi_r(\zeta_{x,1})&=&w_1\alpha (1+b)x-(w_3+w_4n)\\&\sim& w_1(1+b)rx-(w_3+w_4n)=\zeta_{rx,1}.
\end{eqnarray*}\end{proof}

Given an integer $x$ let $l={\rm hcf}(x,2n)$.  There is a positive integer $r$ coprime to $2n$, so that $rx \equiv l$ modulo $2n$.  We conclude:

\begin{lemma}
Any minimal element of $\Omega_3(\Z)$ is isomorphic to $K_{l,1}$ for some positive factor $l$ of $2n$. \label{minisKl}
\end{lemma}

\begin{proof}
We know that a minimal element of $\Omega_3(\Z)$ is isomorphic to $K_{x,1}$ for some integer $x$.  Picking $r,l$ as above, we have:$$K_{x,1}\cong K_{rx,1}\cong K_{l,1}.$$
\end{proof}

\section{Separating the stable classes of the $ K_{l,1}$} \label{stabsep}

Our goal is to show that up to isomorphism there is only one minimal module in $\Omega_3(\Z)$.  In \S\ref{rest} we showed that any such module must be isomorphic to the kernel of some surjective map $M\to \ZC$.  In \S\ref{isom} we proceeded to construct isomorphisms between the kernels of such maps, to the point where we know that any such kernel must be isomorphic to one of the $ K_{l,1}$, for $l$ a positive factor of $2n$.

That is as far as we go in that direction.  We now change tack and instead of constructing isomorphisms between our remaining candidates, we will eliminate all but one of them, by showing that they are not even stably equivalent to $J$ (the minimal element of $\Omega_3(\Z)$ from (\ref{res})).

\begin{lemma}
For each $l$ a positive factor of $2n$, we have an exact sequence{\rm:}
\begin{eqnarray}
0\to K_{l,1}\to\ZD^3 \stackrel{{\partial_2}^{(l)}}\to\ZD^2\stackrel {{\partial_1}^{(l)}}\to \ZD, \label{exactseq}
\end{eqnarray}
where
$${\partial_2}^{(l)}= \left[ \begin{array}{ccc} \Sigma & 0 & l(1+ba) \\ 0 & 1+b & a-1 \end{array} \right], \qquad \qquad{\partial_1}^{(l)}= \left[ \begin{array}{cc} 1-a &l(1-b)\end{array} \right],
$$
with respect to the basis $\{E_1,\, E_2,\,E_3\}$ of $\ZD^3$ and $\{e_1, \, e_2\}$ of $\ZD^2$.
\end{lemma}

\begin{proof}
The component of ${\partial_2}^{(l)}$ mapping into the second summand of $\ZD^2$ is just $\partial_2'$, so it has kernel $M$.  Let $$s\co M \to \ZC$$ denote the component of ${\partial_2}^{(l)}$ mapping into the first summand of $\ZD^2$ restricted to $M$ (again making the identification $\Sigma\ZD\cong \ZC$).  The kernel of ${\partial_2}^{(l)}$ is then the kernel of $s$.
We have:
\begin{eqnarray*}
s(w_1(1-b)) &=&1-b,  
\qquad
\qquad
\,\,\hspace{.1mm} s(w_1(1+b))=1+b,
\\
s(w_2) &=& 0,
\qquad
\qquad
\qquad s(w_3+w_4n)=l(1+b).
\\
s(w_4) &=&0,
\end{eqnarray*}
Thus $s=[(1,0,0), (1,l)]$ and has kernel precisely $K_{l,1}$.

By evaluation we see that ${\partial_1}^{(l)}{\partial_2}^{(l)}=0$.  Let $\left(\begin{array}{c}\nu_1\\ \nu_2 \end{array}\right)$ be an arbitrary element of the kernel of ${\partial_1}^{(l)}$.  We have \begin{eqnarray}
(a-1)\nu_1=l(1-b)\nu_2. \label{aminus1}
\end{eqnarray}
Now $l(1-b)\nu_2=l(1-b)p$ for $p$ some polynomial expression in $a$.  From (\ref{aminus1}) we know $p=(a-1)q$, for $q$ some polynomial expression in $a$.  Now let $$\left(\begin{array}{c}\nu_1'\\ \nu_2' \end{array}\right)=\left(\begin{array}{c}\nu_1\\ \nu_2 \end{array}\right)-\left(\begin{array}{c}l(1+ba)\\ a-1 \end{array}\right)q.$$
Then $(a-1)\nu_1'=l(1-b)\nu_2'=0$ and $\nu_1'=\Sigma \mu_1,\nu_2'=(1+b)\mu_2$, for some $\mu_1,\mu_2 \in \ZD$.  We conclude:$$
\left(\begin{array}{c}\nu_1\\ \nu_2 \end{array}\right)=
\left(\begin{array}{c}l(1+ba)\\ a-1 \end{array}\right)q
+\left(\begin{array}{c}\Sigma\\ 0 \end{array}\right)\mu_1
+\left(\begin{array}{c}0\\ 1+b \end{array}\right)\mu_2,
$$
so the columns of ${\partial_2}^{(l)}$ generate the kernel of ${\partial_1}^{(l)}$.
\end{proof}

Let $I_l\lhd \ZD$ denote the right ideal generated by $1-a, l(1-b)$.  We regard it as a right submodule of $\ZD$.  Note that $I_l$ is $\Z$--torsion free as it is a submodule of $\ZD$.  From (\ref{exactseq}) we have the exact sequence:
$$0\to K_{l,1}\to\ZD^3 \stackrel{{\partial_2}^{(l)}}\to\ZD^2\to I_l\to 0.$$

\begin{lemma}
If $K_{x,1}$ is stably equivalent to $K_{y,1}$ for positive factors $x,y$ of $2n$, then $I_x$ is stably equivalent to $I_y$. \label{stabI}
\end{lemma}

\begin{proof}
Suppose $K_{x,1}\oplus F_1 \cong K_{y,1} \oplus F_2$ for f.g.\!\! free modules $F_1,F_2$.  We have exact sequences: 
\begin{eqnarray*}
0\to K_{x,1}\oplus F_1\to\ZD^3\oplus F_1\stackrel{({\partial_2}^{(x)},0)}\longrightarrow\ZD^2\to I_x\to 0,
\\
0\to K_{y,1}\oplus F_2\to\ZD^3 \oplus F_2 \stackrel{({\partial_2}^{(y)},0)}\longrightarrow\ZD^2\to I_y\to 0.
\end{eqnarray*}
All the modules in these exact sequences are $\Z$--torsion free and finitely generated over $\Z$, so we preserve exactness when we  dualise the sequences by applying the functor ${\rm Hom}_{\Z}(\__,\,\Z)$.  Note this functor takes a right $\ZD$ module $N$ to a right $\ZD$ module $N^*$, with the $\ZD$ action given by $(fg)x=f(xg^{-1})$ for all $f\in N^*, x\in N, g\in D_{4n}$.

Thus by Schanuel's Lemma we have $$I_x^* \oplus F_3 \cong I_y^*\oplus F_4,$$ for f.g.\!\! free modules $F_3,F_4$.  Dualising again we have $I_x$ stably equivalent to $I_y$.
\end{proof}

Clearly ${\partial_2}^{(1)}=\partial_2$, $K_{1,1}=J$, and $I_1=I$, the augmentation ideal in $\ZD$. 

\begin{lemma} Let $y$ be a positive factor of $2n$.  Then $I_y$ is stably equivalent to $I$ if and only if $I_y \cong I$. \label{isoI}
\end{lemma}

\begin{proof}
As $D_{4n}$ is a finite group satisfying the Eichler condition, we have that cancellation of free modules holds in the stable class of $I$ \cite[Proposition 2.7(i)]{John5}.  We then need only note that $I_y$ and $I$ have the same rank over $\Z$. 
\end{proof}

 It remains to show that for $y>1$ a factor of $2n$, the modules $I$ and $I_y$ are not isomorphic and hence not stably equivalent.

\newpage
\begin{lemma}
For $y>1$ a factor of $2n$, the modules $I$ and $I_y$ are not isomorphic.\label{notiso}
\end{lemma}

\begin{proof}
Let $\theta\co I \to I_y$ be an isomorphism.  Note the cokernel of the inclusion $I \hookrightarrow \ZD$ is torsion free and finitely generated.   The module $\ZD$ is strongly injective relative to the tame class of f.g.\!\! torsion free $\ZD$ modules \cite[Lemma 19.5]{John1}.  We therefore have a $\ZD$--linear map $\hat{\theta}\co \ZD \to \ZD$ such that the following diagram commutes:
\begin{center}${}$
\xymatrix{I\ar[r]\ar[d]^\theta& \ZD\ar[d]^{\hat{\theta}}\\
I_y\ar[r]&\ZD}
\end{center}
Then $\hat{\Theta}$ is just left multiplication by some element $p+qb\in\ZD$, where $p,q$ are polynomial expressions in $a$.

Now $\theta$ restricts to an isomorphism between the annihilators of $\Sigma$ in $I$ and $I_y$.  These annihilators of $\Sigma$ are:  $$(a-1)\ZD\lhd I,\qquad\qquad (a-1)\ZD\lhd I_y.$$   Let $\sigma\co (a-1)\ZD \to (a-1)\ZD$ denote the inverse of this restriction.

As before, the cokernel of the inclusion $(a-1)\ZD\hookrightarrow \ZD$ is torsion free and finitely generated.  Thus we have a commutative diagram:
\begin{center}${}$
\xymatrix{(a-1)\ZD\ar[r]\ar[d]^\sigma& \ZD\ar[d]^{\hat{\sigma}}\\
(a-1)\ZD\ar[r]&\ZD}
\end{center}
where $\hat{\sigma}\co \ZD \to \ZD$ is left multiplication by some element $p'+q'b\in \ZD$, with $p',q'$ polynomial expressions in $a$.  

We know that $\hat{\sigma}\hat{\theta}$ restricts to the identity on $(a-1)\ZD$.  Thus $$(p'+q'b)(p+qb)=1+\Sigma\gamma,$$ for some $\gamma\in\ZD$.

Let $t\co\ZD \to \Z$ be the ring homomorphism mapping $a\mapsto 1,\,\, b\mapsto -1$.  We have:$$t(p'+q'b)t(p+qb)\equiv 1 \,{\rm modulo}\, 2n.$$  Thus we have $t(p+qb)$ coprime to $2n$.

Now \begin{eqnarray*}\theta(1-b)&=&(p+qb)(1-b)\\&=&(p-q)(1-b)\\&=&t(p+qb)(1-b)+(a-1)\delta(1-b),
\end{eqnarray*} for some $\delta$ a polynomial expression in $a$.

But $\theta(1-b)\in I_y$, so $y$ divides $t(p+qb)$. This contradicts $t(p+qb)$ coprime to $2n$.
\end{proof}

\noindent  {\bf Theorem A} {\sl Over $\ZD$ any module in $\Omega_3(\Z)$ is  of the form $J \oplus \ZD^r$.}

\begin{proof}
Suppose $K$ is in $\Omega_3(\Z)$ and not isomorphic to $J \oplus \ZD^r$ for any integer $r\geq 0$.  Then $K$ must be a minimal element of $\Omega_3(\Z)$, so by Lemma \ref{minisKl} we know that $K\cong K_{y,1}$ for some $y>1$ a factor of $2n$.  Then by Lemma \ref{stabI} we have that $I_y$ is stably equivalent to $I$, and hence by Lemma \ref{isoI} we know that $I \cong I_y$, contradicting Lemma \ref{notiso}.
\end{proof}

As Johnson has verified the D(2) property for $D_{4n+2}$ \cite{John, John1}, in the light of Theorem \ref{shea} taken from \cite{Shea} we may conclude:

\noindent {\bf Theorem B} {\sl All finite dihedral groups satisfy the D(2) property.}

\begin{proof}
This has been shown for the dihedral groups $D_{4n+2}$ \cite{John,John1}.  Suppose now that $X$ is a cohomologically 2 dimensional finite cell complex with fundamental group $D_{4n}$.  By Theorem A we know that $\pi_2(X)\cong J \oplus \ZD^r$, for some integer $r$.  Hence by Theorem \ref{shea} we know that $X$ is homotopy equivalent to a 2 dimensional finite cell complex. 
\end{proof}

\end{document}